\documentclass[12pt,oneside]{article}
\usepackage{amssymb}

\usepackage[latin1]{inputenc}
\usepackage[english]{babel}
\usepackage{amsmath}

\newtheorem{theorem}{Theorem}[section]

\newtheorem{corollary}[theorem]{Corollary}

\newtheorem{lemma}[theorem]{Lemme}

\newtheorem{proposition}[theorem]{Proposition}

\newenvironment{proof}[1][Proof]{\textbf{#1.} }{\ \rule{0.5em}{0.5em}}

\begin{document}
\title{{\bf
Criteria of convergence of a non ordinary random continued fractions on a symmetric cone }\\
 {\small (Running title: \textbf{random continued fractions})}}
\author {Abdelhamid Hassairi\footnote{
 \textit{E-mail address: Abdelhamid.Hassairi@fss.rnu.tn}}$\;$
\\{\footnotesize{\it
Sfax University, Tunisia.}}\\
    {\footnotesize{\it  }}}
\date{}
 \maketitle

$\overline{\hspace{13cm}}$\vskip0.3cm {\small {\bf Abstract}} { In
this paper, we use a notion of ratio based on a division algorithm,
to extend to a symmetric cone the definition of a continued fraction
in its more general form. We then give a criteria of convergence of
a non ordinary random continued fraction that has arisen in the
study of some probability distributions related
to the beta distribution on the cone of positive definite symmetric matrices or on any symmetric cone.}

$\overline{\hspace{13cm}}$\vskip1cm

\section{Introduction }

 Besides the well known theory of real continued
fractions and its applications, many multivariate versions have been
introduced to bring answers to some needs which have arisen in
different areas. The most important seem to be the ones defined on
some matrix spaces. Since there is no single way to define a matrix
inverse, there exist many ways to define continued fractions with
matrix argument. A first natural way is to assume invertibility and
use the classical matrix inverse to define a ratio of matrices. A
second way is based on the partial matrix inverse. Another way uses
the generalized inverse considered to be more efficient in matrix
continued fraction interpolation problems. For details about these
different matrix continued fractions and their applications, we
refer the reader to $[5]$ and to the references within. It is worth
pointing out that for the applications, the convergence of a matrix
continued fraction is usually needed. However even if in many cases,
a three term recurrence relation is available, because of the
non-commutativity, it is difficult to derive applicable convergence
criteria. The present work is motivated by probability problems
concerning the characterization of some distributions on the cone of
positive symmetric matrices or more generally on any symmetric cone
such as the beta-hypergeometric distribution. The results will be
given in the general setting of a Jordan algebra and its symmetric
cone, however to make the paper accessible for a reader who is not
familiar with the theory of Jordan algebra, we have emphasized on
the case of symmetric matrices. Of course in these spaces, we can't
use the multiplication by the inverse to define a ratio and we need
to use a suitable division algorithm. We mention here that, for the
proof of a characterization result concerning the Wishart
distribution on a symmetric cone, Evelyne Bernadac $[1]$ has used an
ordinary continued fraction defined with a construction process
based on the inversion without any use of multiplication or ratio.
Assuming the independence of the random variables and some
conditions on their distributions, she has given a criteria of
convergence for the ordinary continued fraction. As a continuation
of this work, we first use a notion of quotient based on a division
algorithm to define in its more general form, a continued fraction
$K(x_{n}/y_{n})$ where $(x_{n})$ and $(y_{n})$ are two sequences of
elements of the symmetric cone. The ordinary continued fraction
considered in $[1]$ corresponds to the case where $x_{n}=e$, for all
$n$, where $e$ is the identity element of the algebra. We then show
that, as in the classical matrix case (see $[4]$), any non ordinary
continued fraction is equivalent to an ordinary one. This result has
a mathematical interest, but it has no implication on the
convergence of the non ordinary random continued fraction. In fact,
when the sequences in the non ordinary fraction are constituted of
independent random variables, the corresponding ordinary continued
fraction has not this property. Accordingly, the main part of the
paper is devoted to establish a criteria of convergence for the non
ordinary random continued fractions of the form $K(x_{n}/e)$, where
$(x_{n})$ is a sequence of independent random variables valued in a
symmetric cone. This kind of non ordinary random continued fractions
with the ordinary ones arise in the characterization of many
important distributions on symmetric matrices or more generally on
any symmetric cone.

\section{Symmetric cones}

It is well known that any symmetric cone is associated to a Jordan
algebra, so that in order to present our results in their most
general form, it is necessary to review some facts concerning Jordan
algebra and their symmetric cones. Our notations are the ones used
in $[2]$ or in $[3]$. Let us recall that a Euclidean Jordan algebra
is a Euclidean space $V$ with scalar product $<x,y>$ and a bilinear
map
$$V\times V\longrightarrow V,\ (x,y)\longmapsto x.y$$
called Jordan product such that, for all $x,\ y,\ z$ in $V$,

i) $x.y=y.x$,

ii) $<x,y.z>=<x.y,z>$,

iii) there exists $e$ in $V$ such that $e.x=x$,

iv) $x.(x^2.y)=x^2.(x.y)$, where we used the abbreviation
$x^2=x.x$.\\
A Euclidean Jordan algebra is said to be simple if it does not
contain a nontrivial ideal. \\For the algebra of symmetric matrices,
the Jordan product is defined by $x.y=\frac{1}{2}(xy+yx)$ where $xy$
is the ordinary product of the matrices $x$ and $y$ and the scalar
product is defined by $<x,y>=\textrm{trace}(xy)$.

Now, to each Euclidean simple Jordan algebra $V$, we attach the set
of Jordan squares
$$\overline{\Omega}=\{x^2;\ x\in V\}.$$
Its interior is denoted $\Omega$. \\In general, $\Omega$ is a
symmetric cone, i.e., a convex cone which is

i) self dual, i.e., $\Omega=\{x \in V;\ <x,y>\
>0\ \forall y \in \overline{\Omega}\setminus \{0\} \}$.

ii) homogeneous, i.e., the group $G(\Omega)$ of linear automorphisms
of $\Omega$ acts transitively on $\Omega$.
\\Observe that when $V$
is the algebra of symmetric matrices, $\Omega$ (resp.
$\overline{\Omega}$) denotes the cone of symmetric positive definite
(resp. nonnegative) matrices. \\By analogy with the case where
$V=\Bbb{R}$ and $\Omega=]0,\infty[$, for $x$ and $y$ in the cone
$\Omega$, we write $0<x<y$ if $y-x\in\Omega$. With this, one can
talk about an increasing or decreasing sequence in $\Omega$.\\Given
a Euclidean simple Jordan algebra $V$, we denote by $G$ the
connected component of the identity in $G(\Omega)$. If $V$ is the
space of $(r,r)$-symmetric matrices and $GL(\Bbb{R}^{r})$ is the
group of invertible matrices, The elements of $G(\Omega)$ are the
maps $g:V\longrightarrow V$ such that there exists $a$ in
$GL(\Bbb{R}^{r})$ with
$$g(x)=axa^*,$$
where $a^*$ is the transpose of $a$. \\
For two automorphisms $g_{1}$ and $g_{2}$ in $G$, we write
$g_{1}g_{2}$ the composition of $g_{1}$ and $g_{2}$, and for
$g_{1},...,g_{n}$ in $G$, we write
$\displaystyle\prod_{i=1}^{n}g_{i}=g_{1}...g_{n}$.\\ Throughout, we
will also use the following notations for $(x,y,z)\in V$:

$L(x)y=xy$

$P(x)y=2x(xy)-x^2y$  (quadratic representation)

$x\Box y=L(xy)+[L(x),L(y)]=L(xy)+L(x)L(y)-L(y)L(x)$.

An element $c$ of $V$ is said to be idempotent if $c^2=c$. It is a
primitive idempotent if furthermore $c \neq 0$ and is not the sum
$t+u$ of two non-null idempotents $t$ and $u$ such that $t.u=0$.

A Jordan frame is a set $\{c_1,...,c_r\}$ of primitive idempotents
such that $\displaystyle \sum_{i=1}^rc_i=e$ and
$c_i.c_j=\delta_{ij}c_i$, for $1\leq i,j\leq r$. The size $r$ of
such a frame is a constant called the rank of $V$. For any element
$x$ of a Euclidean simple Jordan algebra, there exist a Jordan frame
$(c_i)_{1\leq i \leq r}$ and $(\lambda_1,...,\lambda_r)$ in
$\Bbb{R}^{r}$ such that $x=\displaystyle\sum_{i=1}^r\lambda_i c_i$.
The real numbers $\lambda_1,\lambda_2,...,\lambda_r$ depend only on
$x$. They are called the eigenvalues of $x$ and this decomposition
is called its spectral decomposition. The trace and the determinant
of $x$ are then respectively defined by
$\textrm{tr}(x)=\displaystyle\sum_{i=1}^r\lambda_i$ and
$\Delta(x)=\displaystyle\prod_{i=1}^r\lambda_i$.

If $c$ is a primitive idempotent of $V$, the only possible
eigenvalues of $L(c)$ are $0,\frac{1}{2}$ and $1$. The corresponding
eigenspaces are respectively denoted by $V(c,0),\ E(V,\frac{1}{2})$
and $V(c,1)$ and the decomposition
$$V=V(c,0)\oplus V(c,\frac{1}{2})\oplus V(c,1)$$
is called the Peirce decomposition of $V$ with respect to $c$.

Suppose now that $(c_i)_{1\leq i \leq r}$ is a Jordan frame in $V$
and for $1\leq i,j \leq r$, let
$$ V_{ij}=\left\{
\begin{array}{l}
V(c_i,1)=\Bbb{R} c_i\ \ \ \ \ \ \ \ \ \ \textrm{if}\ i=j  \\
V(c_i,\frac{1}{2})\cap V(c_j,\frac{1}{2})\ \ \ \ \  \textrm{if}\ ieq
j.
\end{array}
\right. $$ Then (See $[2]$ Th.IV.2.1) we have $V=\underset{i\leq
j}{\displaystyle\oplus }V_{ij}$ and the dimension of $V_{ij}$ is,
for $i\neq j$, a constant $d$ called the Jordan constant. It is
related to the dimension $n$ and the rank $r$ of $V$ by the relation
$n=r+r(r-1)\frac{d}{2}.$ When $V$ is the algebra of symmetric
matrices, $d=1$.\\ If $c$ is an idempotent and if $z$ is an element
of $V(c,\frac{1}{2})$,
$$\tau_c(z)=\exp (2z\Box c)$$
is called a Frobenius transformation. It is an element of the group
$G$.
\\ Given a Jordan frame $(c_i)_{1\leq i \leq r}$, the subgroup of $G$
$$
T=\left\{ \tau _{c_{1}}(z^{(1)})\cdots\tau
_{c_{r-1}}(z^{(r-1)})P\left(\displaystyle\sum_{i=1}^r a_ic_i\right)
,\ a_i>0,\ z^{(j)}\in \displaystyle\bigoplus_{k=j+1}^r
V_{jk}\right\}
$$
is called the triangular group corresponding to the Jordan frame
$(c_i)_{1\leq i \leq r}$.\\ It is an important result (see $[2]$,
p.113) that for an element $y$ in the symmetric cone $\Omega$ there
exists a unique element in the triangular group $T$
such that $y=t(e)$. \\
To define a notion of "quotient" in the symmetric cone $\Omega$, we
usually introduce a division algorithm which is a measurable map $g$
from $\Omega$ into $G$ such that, for all $y\in \Omega$,
$g(y)(y)=e$. With the algorithm $g$, the "quotient" of an element
$x$ by $y$ is defined as $g(y)(x)$. The most usuel division
algorithms are the one corresponding to the quadratic representation
and the one corresponding to the Cholesky decomposition. In fact,
one can define the quotient of $x$ by $y$ as $P(y^{-\frac{1}{2}})x$.
On the other hand, given that for each $y$ in $\Omega $, there
exists a unique $t$ in the triangular group $T$ such that $y=t(e)$,
we define the multiplication of $x$ by $y$ as
$$\pi(y)(x)=t(x),$$ and the map $y\mapsto t^{-1}$ from $\Omega $ into $G $
realizes a division algorithm so that the "quotient" of $x$ by $y$
is defined as
$$\pi^{-1}(y)(x)=t^{-1}(x).$$
We also set
\begin{eqnarray*}
\pi^{*}(y)(x)=t^{*}(x) \ \ and \ \ \pi^{*-1}(y)(x)=t^{*-1}(x).
\end{eqnarray*}
The definition of $\pi(y)$ and $\pi^{*}(y)$ may be extended to
$y\in-\Omega$, by setting $$\pi(-y)=-\pi(y)\ \textrm{and} \
\pi^{*}(-y)=-\pi^{*}(y).$$ Consider the set
$$Str(V)=\{g\in GL(V) \ \ ; P(g(x))=gP(x)g^{*}\}.$$
It is shown in $[2]$ page 150 that $G(\Omega)$ is a subgroup of
$Str(V)$. Hence we have in particular that for $t$ in the triangular
group $T$,
\begin{equation}
P(t(e))=tP(e)t^{*}=tt^{*}.
\end{equation} If $y=t(e)$, then
\begin{equation}\label{272}
P(y)=\pi(y)\pi^{*}(y).
\end{equation}
When $\Omega$ is the cone of positive definite symmetric matrices,
the Cholesky decomposition of an element $y$ of $\Omega$, consists
in writing  $y$ in a unique manner as $ y=tt^{*}$, where $t$ is a
lower triangular matrix with strictly positive diagonal, and we have
$$\pi(y)(x)=txt^{*}  \, \  \pi^{-1}(y)(x)=t^{-1}xt^{*-1}, \ and \ P(y)(x)=yxy,$$
where the last product is the ordinary product of matrices.\\ In
this case, (\ref{272}) is easily established.

\section{Continued fraction in a symmetric cone}
We use the division algorithm defined by the Cholesky decomposition.
Let $(x_{n})_{n\geq1}$ and $(y_{n})_{n\geq0}$ be two sequences in
$\Omega$. The continued fraction denoted $$K(x_{n}/y_{n})$$ or
$$y_{0}+\left[\frac{x_{1}}{y_{1}},\frac{x_{2}}{y_{2}},...\right]$$
 is
an expression whose the $n$th convergent
$$R_{n}=y_{0}+\left[\frac{x_{1}}{y_{1}},...,\frac{x_{n}}{y_{n}}\right]$$
is defined in the following recursive way.\\
$$\left[\frac{x_{1}}{y_{1}}\right]=\left(\pi^{*-1}(x_{1})y_{1}\right)^{-1}=\pi(x_{1})\left(y_{1}^{-1}\right)$$
\begin{eqnarray*}
\left[\frac{x_{1}}{y_{1}},...,\frac{x_{n+1}}{y_{n+1}}\right]&=&
\left(\pi^{*-1}(x_{1})\left(y_{1}+\left[\frac{x_{2}}{y_{2}},...,
\frac{x_{n+1}}{y_{n+1}}\right]\right)\right)^{-1}\\&=&\pi(x_{1})\left(y_{1}+\left[\frac{x_{2}}{y_{2}},...,
\frac{x_{n+1}}{y_{n+1}}\right]\right)^{-1}.
\end{eqnarray*} Note
that to get $R_{n+1}$ from $R_{n}$, we replace $y_{n}$ by
$y_{n}+\left(\pi^{*-1}(x_{n+1})y_{n+1}\right)^{-1}$. \\

When $x_{n}=e$, for all $n$, we say that the continued fraction
$K(e/y_{n})$ is an ordinary continued fraction. Its $k$th convergent
is given by
\begin{equation} \label{76}
R_{k}=y_{0}+\left(y_{1}+\left(y_{2}+\left(y_{3}+...+\left(y_{k}\right)^{-1}\right)^{-1}..
.\right)^{-1}\right)^{-1}.
\end{equation}
As mentioned in the introduction, the construction process of this
continued fraction uses only the inversion without any use of the
product or of the division algorithm.
\\Next, we show that any continued fraction $K(x_{n}/y_{n})$ is
equivalent to an ordinary continued fraction $K(e/a_{n})$.
\begin{proposition} The continued fraction denoted $K(x_{n}/y_{n})$
is equivalent to the continued fraction $K(e/a_{n})$ where

$$a_{0}=y_{0}\ \ \ \ \ \ \ \ \ \ \ \ \ \ \ \ \ \ \ \ \ \ \ \ \ \ \ \ \ \ \ \ \ \ \ \ \ \ \ \ \ \ \ \ \ \ \ \ \ \ \   $$
\begin{equation} \label{76}
a_{2k}=\pi(x_{1})\pi^{*-1}(x_{2})...\pi(x_{2k-1})\pi^{*-1}(x_{2k})y_{2k}
\ \ \ \ \ \
\end{equation}
\begin{equation} \label{78}
a_{2k+1}=\pi^{*-1}(x_{1})\pi(x_{2})...\pi(x_{2k})\pi^{*-1}(x_{2k+1})y_{2k+1}.
\end{equation}
\end{proposition}
\begin{proof}
The proof is performed by induction. We will use the fact that
\begin{equation} \label{87}
\pi^{-1}(u)v^{-1}=(\pi^{*}(u)v)^{-1}\ \ \textrm{or \ equivalently} \
\ \pi^{*}(u)v^{-1}=(\pi^{-1}(u)v)^{-1}
\end{equation}
 Since
$$\left[\frac{x_{1}}{y_{1}}\right]=\left(\pi^{*-1}(x_{1})y_{1}\right)^{-1},$$
and
\begin{eqnarray*}
\left[\frac{x_{1}}{y_{1}},\frac{x_{2}}{y_{2}}\right]&=&\left(\pi^{*-1}(x_{1})\left(y_{1}+\left[\frac{x_{2}}{y_{2}}\right]\right)\right)^{-1}\\
&=&\left(\pi^{*-1}(x_{1})\left(y_{1}+\left(\pi^{*-1}\left(x_{2}\right)y_{2}\right)^{-1}\right)\right)^{-1}\\&=&
\left(\pi^{*-1}(x_{1})y_{1}+\pi^{*-1}(x_{1})\left(\pi^{*-1}\left(x_{2}\right)y_{2}\right)^{-1}\right)^{-1}\\&=&
\left(\pi^{*-1}(x_{1})y_{1}+\left(\pi(x_{1})\pi^{*-1}\left(x_{2}\right)y_{2}\right)^{-1}\right)^{-1}.
\end{eqnarray*}
We have that
\begin{eqnarray*}
a_{0}&=&y_{0}\\
a_{1}&=&\pi^{*-1}(x_{1})y_{1}\\
a_{2}&=&\pi(x_{1})\pi^{*-1}\left(x_{2}\right)y_{2}.
\end{eqnarray*}
Therefore (\ref{78}) is true for $k=0$ and (\ref{76}) is true for
$k=1$.\\
Suppose that
$$R_{2k}=a_{0}+\left(a_{1}+\left(a_{2}+\left(a_{3}+...+\left(a_{2k}\right)^{-1}\right)^{-1}...\right)^{-1}\right)^{-1}.$$
with the $a_{i}$, $0\leq i\leq 2k$, defined by (\ref{78}) and
(\ref{76}). In particular
$$a_{2k}=\pi(x_{1})\pi^{*-1}(x_{2})...\pi(x_{2k-1})\pi^{*-1}(x_{2k})y_{2k}.$$
To get $R_{2k+1}$ from $R_{2k}$, we replace in $a_{2k}$, $y_{2k}$ by
$y_{2k}+\left(\pi^{*-1}(x_{2k+1})y_{2k+1}\right)^{-1}$, that is we
replace $a_{2k}$ by
$$\left(a_{2k}+\pi(x_{1})\pi^{*-1}(x_{2})...\pi(x_{2k-1})\pi^{*-1}(x_{2k})\left(\pi^{*-1}(x_{2k+1})y_{2k+1}\right)^{-1}\right)^{-1}.$$
Using (\ref{87}), this can be written
$$\left(a_{2k}+\left(\pi^{*-1}(x_{1})\pi(x_{2})...\pi^{*-1}(x_{2k-1})\pi(x_{2k})\pi^{*-1}(x_{2k+1})y_{2k+1}\right)^{-1}\right)^{-1}.$$
Hence
$$R_{2k+1}=a_{0}+\left(a_{1}+\left(a_{2}+\left(a_{3}+...+\left(a_{2k}+\left(a_{2k+1}\right)^{-1}\right)^{-1}\right)^{-1}...\right)^{-1}\right)^{-1}.$$
with
$$a_{2k+1}=\pi^{*-1}(x_{1})\pi(x_{2})...\pi^{*-1}(x_{2k-1})\pi(x_{2k})\pi^{*-1}(x_{2k+1})y_{2k+1}.$$
A similar reasoning is used for $R_{2k+2}$.
\end{proof}
\section{Convergence of a non ordinary continued fraction}
In this section, we suppose that $(x_{n})_{n\geq1}$ and
$(y_{n})_{n\geq0}$ are sequences of independent random variables
defined on the same probability space and valued in the cone
$\Omega$. As we have mentioned, a criteria of convergence for the
ordinary random continued fraction $K(e/y_{n})$ has been given in
$[1]$. In what follows, we give a criteria of convergence for the
non ordinary random continued fraction $K(x_{n}/e)$. For the sake of
simplicity,
$\left[\frac{x_{1}}{e},\frac{x_{2}}{e},...,\frac{x_{k}}{e}\right]$
will be written $\left[x_{1},x_{2},...,x_{k}\right]$ which is the
notation usually used for the ordinary continued fraction given in
(\ref{76}). This should not cause any confusion, the continued
fractions are completely different, the results and their proofs are
completely different although there are some similarities. The
continued fraction $\left[x_{1},x_{2},...,x_{k}\right]$ which we
will study in what follows is then defined by
\begin{equation}\label{384}
\left[x_{1}\right]=x_{1} \ \ \textrm{and} \ \
\left[x_{1},x_{2},...,x_{k}\right]=\pi(x_{1})(e+\left[x_{2},...,x_{k}\right])^{-1}.
\end{equation}
We first prove the following result:
\begin{proposition}Let $(x_{n})_{n\geq1}$ be a
sequence in $\Omega$. Then for $k\geq1$,
$$(-1)^{k+1}\left(\left[x_{1},x_{2},...,x_{k}\right]-\left[x_{1},x_{2},...,x_{k+1}\right]\right)$$
is in $\Omega$.
\end{proposition}
\begin{proof} The proof uses induction on $k$.\\
The property is true for $k=1$, in fact
\begin{eqnarray*}
\left[x_{1}\right]-\left[x_{1},x_{2}\right]&=&x_{1}-\pi(x_{1})\left(e+x_{2}\right)^{-1}\\
&=&\pi(x_{1})\left(e-\left(e+x_{2}\right)^{-1}\right)\\
&=&\pi(x_{1})\left((e+x_{2}-e)(e+x_{2})^{-1}\right),
\end{eqnarray*}
which is in $(-1)^{2}\Omega$. Suppose that the property is true for
$k$. Then
\begin{eqnarray*}
\left[x_{1},...,x_{k+1}\right]-\left[x_{1},...,x_{x+2}\right]&=&\pi(x_{1})\left(e+\left[x_{2},...,x_{k+1}\right]\right)^{-1}
-\pi(x_{1})\left(e+\left[x_{2},...,x_{k+2}\right]\right)^{-1}\\
&=&\pi(x_{1})\left(\left(e+\left[x_{2},...,x_{k+1}\right]\right)^{-1}
-\left(e+\left[x_{2},...,x_{k+2}\right]\right)^{-1}\right)\\
\end{eqnarray*}
Using the induction hypothesis, we have that
$$\left[x_{2},...,x_{k+1}\right]-\left[x_{2},...,x_{k+2}\right]\in(-1)^{k+1}\Omega,$$
so that
$$\left(e+\left[x_{2},...,x_{k+1}\right]\right)-\left(e+\left[x_{2},...,x_{k+2}\right]\right)\in(-1)^{k+1}\Omega.$$
To finish the proof, we use the fact that for all $x$ and $y$ in
$\Omega$, $$\left(x-y\in\Omega\right) \ \ \Leftrightarrow \ \
\left(x^{-1}-y^{-1}\in -\Omega\right).$$
\end{proof}\\
Now, Denoting
\begin{equation}\label{461}
w_{k}=(-1)^{k+1}\left(\left[x_{1},x_{2},...,x_{k}\right]-\left[x_{1},x_{2},...,x_{k+1}\right]\right),
\end{equation}
it is easy to see that
\begin{equation}
\left[x_{1},x_{2},...,x_{n}\right]=x_{1}+\displaystyle\sum_{k=1}^{n-1}(-1)^{k}w_{k}.
\end{equation}
Thus the continued fraction $\left[x_{1},x_{2},...,x_{n}\right]$
converges if and only if the alternating series
$\displaystyle\sum_{k=1}^{+\infty}(-1)^{k}w_{k}$ converges. For the
convergence of an alternating series in the algebra $V$, we have the
following result which appears in $[1]$.
\begin{proposition}Let $(u_{n})_{n\geq1}$ be a
sequence in $\Omega$. If $(u_{n})$ is decreasing and converges to
$0$, then the alternating series
$\displaystyle\sum_{k=1}^{+\infty}(-1)^{k}u_{k}$ is convergent
\end{proposition}
Next, we state and prove a theorem which is crucial for the study of
the convergence of the alternating series
$\displaystyle\sum_{k=1}^{+\infty}(-1)^{k}w_{k}$.
\begin{theorem} \label{483}Let $(x_{n})_{n\geq1}$ be a
sequence in $\Omega$, and denote
\begin{equation}\label{486}
\begin{array}{c}
F_{k}(x_{1},...,x_{k+2})=\left(\left[x_{1},...,x_{k}\right]^{-1}-\left[x_{1},...,x_{k+1}\right]^{-1}\right)^{-1}\\
\ \ \ \ \ \ \ \ \ \ \ \ \ \ \ \ \ \ \ \ \ \ \ \ +
\left(\left[x_{1},...,x_{k+1}\right]^{-1}-\left[x_{1},...,x_{k+2}\right]^{-1}\right)^{-1}.
\end{array}
\end{equation}
Then we have
\begin{equation}\label{492}
F_{1}=\pi(x_{1})\pi^{*-1}(x_{2})\pi^{*-1}(x_{3})(e),
\end{equation}
\begin{equation}\label{444}
F_{2}(x_{1},x_{2},x_{3},x_{4})=-\pi(x_{1})\pi^{*-1}(x_{2})\pi^{*-1}
\left(x_{3}\right)P\left(e+\pi^{*}(x_{3})(e)\right)\pi^{*-1}(x_{4})(e),
\end{equation}
 and for all
$k\geq3$,
\begin{equation}\label{497}
\begin{array}{c}
F_{k}(x_{1},...,x_{k+2})=(-1)^{k+1}\pi(x_{1})\left(\displaystyle\prod_{i=2}^{k-1}\pi^{*-1}(x_{i})P\left(e+\left[x_{i+1},...,x_{k}\right]\right)\right)
\\
\ \ \ \ \ \ \ \ \ \ \ \ \ \ \ \ \ \ \ \ \ \ \ \
 \pi^{*-1}(x_{k})\pi^{*-1}(x_{k+1})P\left(u_{k}\right)\pi^{*-1}(x_{k+2})(e),
\end{array}
\end{equation}
where
\begin{equation}\label{459}
\begin{array}{c}
u_{k}=u_{k}(x_{1},...,x_{k+1}) =e+\pi^{*}(x_{k+1})\ \ \ \ \ \ \ \ \ \ \ \ \ \ \ \ \ \ \ \ \ \ \ \ \ \ \ \ \ \ \ \ \ \ \ \ \ \ \ \ \ \ \ \\
\left(e+\displaystyle\sum_{i=0}^{k-3}(-1)^{i+1}
\left(\displaystyle\prod_{j=0}^{i}\left(\pi^{*}(x_{k-j})
P\left(\left(e+\left[x_{k-j},...,x_{k}\right]\right)^{-1}\right)\right)\right)
\left(e+\left[x_{k-i},...,x_{k}\right]\right)\right)
\end{array}
\end{equation}
\end{theorem}
\begin{proof}
\begin{eqnarray*}
F_{1}&=&F_{1}(x_{1},x_{2},x_{3})\\
&=&\left(\left[x_{1}\right]^{-1}-\left[x_{1},x_{2}\right]^{-1}\right)^{-1}
+
\left(\left[x_{1},x_{2}\right]^{-1}-\left[x_{1},x_{2},x_{3}\right]^{-1}\right)^{-1}\\
&=&
\left(\pi^{*-1}(x_{1})(e)-\pi^{*-1}(x_{1})(e+x_{2})\right)^{-1}\\&+&
\left(\pi^{*-1}(x_{1})(e+x_{2})-\pi^{*-1}(x_{1})\left(e+\left(\pi^{*-1}(x_{2})(e+x_{3})\right)^{-1}\right)\right)^{-1}
\\
&=& \left(-\pi^{*-1}(x_{1})(x_{2})\right)^{-1}+
\left(\pi^{*-1}(x_{1})(x_{2})-\pi^{*-1}(x_{1})\pi(x_{2})\left((e+x_{3})^{-1}\right)\right)^{-1}\\
&=& -\pi(x_{1})\pi^{*-1}(x_{2})(e)+
\pi(x_{1})\pi^{*-1}(x_{2})\left(\left(e-(e+x_{3})^{-1}\right)^{-1}\right)\\
&=&\pi(x_{1})\pi^{*-1}(x_{2})\left(\left(e-(e+x_{3})^{-1}\right)^{-1}-e\right)\\
&=&\pi(x_{1})\pi^{*-1}(x_{2})\pi^{*-1}(x_{3})(e)
\end{eqnarray*}
Now to calculate
\begin{eqnarray*}
F_{2}=&=&F_{2}(x_{1},x_{2},x_{3},x_{4})\\
&=&\left(\left[x_{1},x_{2}\right]^{-1}-\left[x_{1},x_{2},x_{3}\right]^{-1}\right)^{-1}
+
\left(\left[x_{1},x_{2},x_{3}\right]^{-1}-\left[x_{1},x_{2},x_{3},x_{3}\right]^{-1}\right)^{-1},
\end{eqnarray*}
we write $$\left[x_{1},x_{2}\right]=\left[X_{1}\right]$$ with
$$X_{1}=\pi(x_{1})\left(e+x_{2}\right)^{-1},$$
and we write
$$\left[x_{1},x_{2},x_{3}\right]=\left[X_{1},X_{2}\right]$$ with
$$X_{2}=-\pi^{*}\left(\left(e+x_{2}\right)^{-1}\right)\pi(x_{2})\left(x_{3}\left(e+x_{3}\right)^{-1}\right),$$
finally, we write
 $$\left[x_{1},x_{2},x_{3},x_{4}\right]=\left[X_{1},X_{2},X_{3}\right],$$
 with
$$X_{3}=\pi^{-1}\left(e+\pi^{*}(x_{3})(e)\right)(x_{4}).$$
Thus using (\ref{492}), we have
\begin{eqnarray*}
F_{2}(x_{1},x_{2},x_{3},x_{4}) &=&
F_{1}(X_{1},X_{2},X_{3})\\&=&\pi(X_{1})\pi^{*-1}(X_{2})\pi^{*-1}(X_{3})(e).
\end{eqnarray*}
Replacing, we get
\begin{eqnarray*}
F_{2}(x_{1},x_{2},x_{3},x_{4})&=&-\pi(x_{1})\pi\left((e+x_{2})^{-1}\right)
\pi^{-1}\left(\left(e+x_{2}\right)^{-1}\right)\pi^{*-1}(x_{2})\\&&\pi^{*-1}
\left(x_{3}\right)\pi\left(e+\pi^{*}(x_{3})(e)\right)
\pi^{*}\left(e+\pi^{*}(x_{3})(e)\right)\pi^{*-1}(x_{4})(e).
\end{eqnarray*}
Given that
$$\pi\left(e+\pi^{*}(x_{3})(e)\right)
\pi^{*}\left(e+\pi^{*}(x_{3})(e)\right)=P\left(e+\pi^{*}(x_{3})(e)\right),$$
we obtain
$$F_{2}(x_{1},x_{2},x_{3},x_{4})=-\pi(x_{1})\pi^{*-1}(x_{2})\pi^{*-1}
\left(x_{3}\right)P\left(e+\pi^{*}(x_{3})(e)\right)\pi^{*-1}(x_{4})(e).$$
Therefore (\ref{497}) is true for $k=2$. Suppose that it true for
$k$. Then with the same reasoning, we have that
$$F_{k+1}=F_{k+1}(x_{1},...,x_{k-1},x_{k},x_{k+1},x_{k+2},x_{k+3})=F_{k}(x_{1},...,x_{k-1},X_{k},X_{k+1},X_{k+2}),$$
with
$$X_{k}=\pi(x_{k})\left(\left(e+x_{k+1}\right)^{-1}\right),$$
$$X_{k+1}=-\pi^{*}\left(\left(e+x_{k+1}\right)^{-1}\right)\pi(x_{k+1})\left(x_{k+2}\left(e+x_{k+2}\right)^{-1}\right),$$
and
$$X_{k+2}=\pi^{-1}\left(e+\pi^{*}(x_{k+2})(e)\right)(x_{k+3}).$$
Inserting this in $F_{k}(x_{1},...,x_{k-1},X_{k},X_{k+1},X_{k+2})$
as defined in (\ref{497}), we have in particular that
$\left[x_{i+1},...,x_{k-1},X_{k}\right]$ is nothing but
$\left[x_{i+1},...,x_{k},x_{k+1}\right]$. We also have
\begin{eqnarray*}
\pi^{*-1}(X_{k})\pi^{*-1}(X_{k+1})&=&-\pi^{*-1}(x_{k})\pi^{*-1}\left(\left(e+x_{k+1}\right)^{-1}\right)
\pi^{-1}\left(\left(e+x_{k+1}\right)^{-1}\right)\\
&&\pi^{*-1}(x_{k+1})\pi^{*-1}\left(x_{k+2}\right)\pi\left(e+\pi^{*}(x_{k+2})(e)\right)\\
&=&-\pi^{*-1}(x_{k})P\left(e+x_{k+1}\right)
\pi^{*-1}(x_{k+1})\pi^{*-1}\left(x_{k+2}\right)\\
&&\pi\left(e+\pi^{*}(x_{k+2})(e)\right),
\end{eqnarray*}
and that
$$\pi^{*-1}(X_{k+2})(e)=\pi^{*}\left(e+\pi^{*}(x_{k+2})(e)\right)\pi^{*-1}(x_{k+3})(e)$$
Thus
\begin{eqnarray*}
F_{k+1}&=&F_{k+1}(x_{1},...,x_{k-1},x_{k},x_{k+1},x_{k+2},x_{k+3})\\
&=&F_{k}(x_{1},...,x_{k-1},X_{k},X_{k+1},X_{k+2})\\
&=&(-1)^{k+1}\pi(x_{1})\left(\displaystyle\prod_{i=2}^{k}\pi^{*-1}(x_{i})P\left(e+\left[x_{i+1},...,x_{k+1}\right]\right)\right)
\\&&
\pi^{*-1}(x_{k+1})\pi^{*-1}\left(x_{k+2}\right)\\
&&\pi\left(e+\pi^{*}(x_{k+2})(e)\right)
P\left(u_{k}(x_{1},...,x_{k-1},X_{k},X_{k+1})\right)\\&&\pi^{*}
\left(e+\pi^{*}(x_{k+2})(e)\right)\pi^{*-1}(x_{k+3})(e).
\end{eqnarray*}
Denoting
$$g=\pi\left(e+\pi^{*}(x_{k+2})(e)\right)
P\left(u_{k}(x_{1},...,x_{k-1},X_{k},X_{k+1})\right)\pi^{*}
\left(e+\pi^{*}(x_{k+2})(e)\right),$$ It remains to verify that
\begin{eqnarray*}
g=P\left(u_{k+1}(x_{1},...,x_{k-1},x_{k},x_{k+1},x_{k+2})\right).
\end{eqnarray*}
In fact,
\begin{eqnarray*}
&& u_{k}(x_{1},...,x_{k-1},X_{k},X_{k+1})=\\
&&e-
\pi^{-1}\left(e+\pi^{*}(x_{k+2})(e)\right)\pi^{*}\left(x_{k+2}\right)\pi^{*}(x_{k+1})
\pi\left(\left(e+x_{k+1}\right)^{-1}\right)\\&& \left(e-\pi^{*}
\left(\left(e+x_{k+1}\right)^{-1}\right)\pi^{*}(x_{k})P\left(\left(e+\left[x_{k},x_{k+1}\right]\right)^{-1}\right)
\left(e+\left[x_{k},x_{k+1}\right]\right)\right.+\\&&
\displaystyle\sum_{i=1}^{k-3}(-1)^{i+1}
\pi^{*}\left(\left(e+x_{k+1}\right)^{-1}\right)\pi^{*}\left(x_{k}\right)
P\left(\left(e+\left[x_{k},x_{k+1}\right]\right)^{-1}\right)
\\&&\left.\left(\displaystyle\prod_{j=1}^{i}\pi^{*}(x_{k-j})
P\left(\left(e+\left[x_{k-j},...,x_{k+1}\right]\right)^{-1}\right)\right)
\left(e+\left[x_{k-i},...,x_{k+1}\right]\right)\right).
\end{eqnarray*}
Using the fact that
$$\pi(y)P(x)\pi^{*}(y)=P(\pi(y)x), \ \ and \ \ \pi^{*}(x^{-1})(x)=e$$
we obtain that
\begin{eqnarray*}
g&=&P \left(e+\pi^{*}(x_{k+2})(e)-
\pi^{*}\left(x_{k+2}\right)\pi^{*}(x_{k+1})
\pi\left(\left(e+x_{k+1}\right)^{-1}\right)\right.\\&&
\left(e-\pi^{*}
\left(\left(e+x_{k+1}\right)^{-1}\right)\pi^{*}(x_{k})P\left(\left(e+\left[x_{k},x_{k+1}\right]\right)^{-1}\right)
\left(e+\left[x_{k},x_{k+1}\right]\right)\right.+\\&&
\displaystyle\sum_{i=1}^{k-3}(-1)^{i+1}
\pi^{*}\left(\left(e+x_{k+1}\right)^{-1}\right)\pi^{*}\left(x_{k}\right)
P\left(\left(e+\left[x_{k},x_{k+1}\right]\right)^{-1}\right)
\\&&\left.\left(\displaystyle\prod_{j=1}^{i}\pi^{*}(x_{k-j})
P\left(\left(e+\left[x_{k-j},...,x_{k+1}\right]\right)^{-1}\right)\right)\left.
\left(e+\left[x_{k-i},...,x_{k+1}\right]\right)\right)\right)\\
&=&P \left(e+\pi^{*}(x_{k+2})(e)-
\pi^{*}\left(x_{k+2}\right)\pi^{*}(x_{k+1})P\left(\left(e+x_{k+1}\right)^{-1}\right)
\left(\left(e+x_{k+1}\right)\right)\right.\\&&
+\pi^{*}\left(x_{k+2}\right)\pi^{*}(x_{k+1})P\left(\left(e+x_{k+1}\right)^{-1}\right)
\pi^{*}(x_{k})P\left(\left(e+\left[x_{k},x_{k+1}\right]\right)^{-1}\right)
\left(e+\left[x_{k},x_{k+1}\right]\right)+\\&&
\displaystyle\sum_{i=1}^{k-3}(-1)^{i+2}
\pi^{*}\left(x_{k+2}\right)\pi^{*}(x_{k+1})P\left(\left(e+x_{k+1}\right)^{-1}\right)\pi^{*}\left(x_{k}\right)
P\left(\left(e+\left[x_{k},x_{k+1}\right]\right)^{-1}\right)
\\&&\left.\left(\displaystyle\prod_{j=1}^{i}\pi^{*}(x_{k-j})
P\left(\left(e+\left[x_{k-j},...,x_{k+1}\right]\right)^{-1}\right)\right)\left.
\left(e+\left[x_{k-i},...,x_{k+1}\right]\right)\right)\right).
\end{eqnarray*}
Setting $i'=i+1$ and $j'=j+1$, with some standard calculation, this
is noting but
$P\left(u_{k+1}(x_{1},...,x_{k-1},x_{k},x_{k+1},x_{k+2})\right)$.
\end{proof}
We now coin a useful lemma.
\begin{lemma}\label{613} For all
$x_{1},...,x_{k}$ in $\Omega$, the vector
$$H=e+\displaystyle\sum_{i=0}^{k-3}(-1)^{i+1}
\left(\displaystyle\prod_{j=0}^{i}\left(\pi^{*}(x_{k-j})
P\left(\left(e+\left[x_{k-j},...,x_{k}\right]\right)^{-1}\right)\right)\right)
\left(e+\left[x_{k-i},...,x_{k}\right]\right)$$ is in $\Omega$.
\end{lemma}Denoting for $1\leq i\leq k-3$,
$$A_{i}=\left(\displaystyle\prod_{j=0}^{i}\left(\pi^{*}(x_{k-j})
P\left(\left(e+\left[x_{k-j},...,x_{k}\right]\right)^{-1}\right)\right)\right)
\left(e+\left[x_{k-i},...,x_{k}\right]\right),$$ we have that
$$H=e-\pi^{*}(x_{k}) P\left(\left(e+x_{k}\right)^{-1}\right)
\left(e+x_{k}\right)+ \displaystyle\sum_{1\leq i\leq k-3, \ i \
odd}(A_{i}-A_{i+1}),$$ with $A_{k-2}=0$. Now
\begin{eqnarray*}
e-\pi^{*}(x_{k}) P\left(\left(e+x_{k}\right)^{-1}\right)
\left(e+x_{k}\right)&=&e-\pi^{*}(x_{k})
\left(\left(e+x_{k}\right)^{-1}\right)\\&=&e -\left(\pi^{-1}(x_{k})
\left(e+x_{k}\right)\right)^{-1}\\&=&e -\left(\pi^{-1}(x_{k})
(e)+e\right)^{-1}\in\Omega.
\end{eqnarray*}
On the other hand,  we have for $i$ odd,
\begin{eqnarray*}A_{i}-A_{i+1}&=&\left(\displaystyle\prod_{j=0}^{i}\left(\pi^{*}(x_{k-j})
P\left(\left(e+\left[x_{k-j},...,x_{k}\right]\right)^{-1}\right)\right)\right)\\
&&\left(\left(e+\left[x_{k-i},...,x_{k}\right]\right)-\pi^{*}(x_{k-(i+1)})
\left(\left(e+\left[x_{k-(i+1)},...,x_{k}\right]\right)^{-1}\right)
\right)\\
&=&\left(\displaystyle\prod_{j=0}^{i}\left(\pi^{*}(x_{k-j})
P\left(\left(e+\left[x_{k-j},...,x_{k}\right]\right)^{-1}\right)\right)\right)\\
&&\left(\left(e+\left[x_{k-i},...,x_{k}\right]\right)-\pi^{*}(x_{k-(i+1)})
\left(\left(e+\pi(x_{k-(i+1)})\left(\left(e+\left[x_{k-i},...,x_{k}\right]\right)^{-1}\right)\right)^{-1}\right)
\right)\\
&=&\left(\displaystyle\prod_{j=0}^{i}\left(\pi^{*}(x_{k-j})
P\left(\left(e+\left[x_{k-j},...,x_{k}\right]\right)^{-1}\right)\right)\right)\\
&&\left(\left(e+\left[x_{k-i},...,x_{k}\right]\right)-
\left(\left(\pi^{-1}(x_{k-(i+1)})(e)+\left(e+\left[x_{k-i},...,x_{k}\right]\right)^{-1}\right)^{-1}\right)
\right),
\end{eqnarray*}
which is clearly in $\Omega$. Thus $H$ is in $\Omega$. !
\begin{theorem} The sequence $(w_{k})$ defined in (\ref{461}) is decreasing.
\end{theorem}
\begin{proof}
Form the definition, we have that
$$\left[e,x_{1},x_{2},...,x_{k}\right]=\pi(e)\left(\left(e+\left[x_{1},x_{2},...,x_{k}\right]\right)^{-1}\right),$$
so that
$$\left[e,x_{1},x_{2},...,x_{k}\right]^{-1}=e+\left[x_{1},x_{2},...,x_{k}\right].$$
Hence
\begin{eqnarray*}
w_{k}^{-1}&=&(-1)^{k+1}\left(\left[x_{1},x_{2},...,x_{k}\right]-\left[x_{1},x_{2},...,x_{k+1}\right]\right)^{-1}\\
&=&(-1)^{k+1}\left(\left[e,x_{1},x_{2},...,x_{k}\right]^{-1}-\left[e,x_{1},x_{2},...,x_{k+1}\right]^{-1}\right)^{-1}
\end{eqnarray*}
\begin{eqnarray*}
w_{k}^{-1}-w_{k+1}^{-1}&=&(-1)^{k+1}\left(\left[e,x_{1},x_{2},...,x_{k}\right]^{-1}-\left[e,x_{1},x_{2},...,x_{k+1}\right]^{-1}\right)^{-1}\\
&-&(-1)^{k+2}\left(\left[e,x_{1},x_{2},...,x_{k+1}\right]^{-1}-\left[e,x_{1},x_{2},...,x_{k+2}\right]^{-1}\right)^{-1}\\
&=&(-1)^{k+1}\left\{\left(\left[e,x_{1},x_{2},...,x_{k}\right]^{-1}-\left[e,x_{1},x_{2},...,x_{k+1}\right]^{-1}\right)^{-1}\right.\\
&&+\left.\left(\left[e,x_{1},x_{2},...,x_{k+1}\right]^{-1}-\left[e,x_{1},x_{2},...,x_{k+2}\right]^{-1}\right)^{-1}\right\} \\
&=&(-1)^{k+2}F_{k+1}(e,x_{1},...,x_{k+2}) \\
&=&
(-1)^{k+1}(-1)^{k+2}\left(\displaystyle\prod_{i=1}^{k-1}\pi^{*-1}(x_{i})P\left(e+\left[x_{i+1},...,x_{k}\right]\right)\right)
\\&&\pi^{*-1}(x_{k})\pi^{*-1}(x_{k+1})P\left(u_{k+1}(e,x_{1},...,x_{k+1})\right)\pi^{*-1}(x_{k+2})(e)\\
\end{eqnarray*}
Therefore
\begin{equation}\label{608}
\begin{array}{c}
w_{k}^{-1}-w_{k+1}^{-1}=
-\left(\displaystyle\prod_{i=1}^{k-1}\pi^{*-1}(x_{i})P\left(e+\left[x_{i+1},...,x_{k}\right]\right)\right)
\pi^{*-1}(x_{k})\pi^{*-1}(x_{k+1})\\P\left(u_{k+1}(e,x_{1},...,x_{k+1})\right)\pi^{*-1}(x_{k+2})(e).
\end{array}
\end{equation}
As for all $v$ in the cone $\Omega$, the automorphisms $\pi(v)$, and
$P(v)$ are in the group $G$, they conserve $\Omega$, we deduce that
$$w_{k}^{-1}-w_{k+1}^{-1}\in-\Omega,$$
which is equivalent to
$$w_{k}-w_{k+1}\in\Omega.$$
Hence the sequence $(w_{k})$ is decreasing.
\end{proof}\\
Next, we give an other important intermediary result. Using the
notation above, for a sequence $(x_{n})$ in the cone $\Omega$, we
define the following element of $G$:
\begin{equation}\label{624}
Q_{k}=
\left(\displaystyle\prod_{i=1}^{k-1}\pi^{*-1}(x_{i})P\left(e+\left[x_{i+1},...,x_{k}\right]\right)\right)
\pi^{*-1}(x_{k})\pi^{*-1}(x_{k+1})P\left(u_{k+1}(e,x_{1},...,x_{k+1})\right)
\end{equation}
Note that according to (\ref{608}), we have
\begin{equation}\label{688}
w_{k+1}^{-1}-w_{k}^{-1}= Q_{k}(x_{k+2}^{-1}).
\end{equation}
\begin{proposition}Let $y$
be an element of $\Omega$. Then for all $k\geq2$,
$$Q_{k}^{*}(y)\geq\pi^{-1}(x_{1})(y)$$
\end{proposition}
\begin{proof}We have that
\begin{eqnarray*}
Q_{k}^{*}=P\left(u_{k+1}(e,x_{1},...,x_{k+1})\right)\pi^{-1}(x_{k+1})\pi^{-1}(x_{k})
\left(\displaystyle\prod_{i=1}^{k-1}P\left(e+\left[x_{k+1-i},...,x_{k}\right]\right)\pi^{-1}(x_{k-i})\right)
\end{eqnarray*}
Using the fact that $P(x)=\pi(x)\pi^{*}(x)$, we obtain that
$Q_{k}^{*}$ is equal to the composition of the following elements of
$G$:
\begin{equation}\label{714}
\pi^{*}\left(e+\left[x_{2},...,x_{k}\right]\right)\pi^{-1}(x_{1}),
\end{equation}
then the automorphisms
\begin{equation}\label{709}
\pi^{*}\left(e+\left[x_{k+1-i},...,x_{k}\right]\right)
\pi^{-1}(x_{k-i})\pi\left(e+\left[x_{k-i},...,x_{k}\right]\right),
\end{equation}
for $i=1,...,k-2$, then
\begin{equation}\label{702}
\pi^{-1}(x_{k}) \pi\left(e+\left[x_{k}\right]\right),
\end{equation}
and
\begin{equation}\label{679}
P\left(u_{k+1}(e,x_{1},...,x_{k+1})\right)\pi^{-1}(x_{k+1}).
\end{equation}
For the terms in (\ref{709}), we have
\begin{eqnarray*}
&&\pi^{*}\left(e+\left[x_{k+1-i},...,x_{k}\right]\right)
\pi^{-1}(x_{k-i})\pi\left(e+\left[x_{k-i},...,x_{k}\right]\right)\\&=&
\pi^{*}\left(e+\left[x_{k+1-i},...,x_{k}\right]\right)
\pi^{-1}(x_{k-i})\pi\left(e+\pi(x_{k-i})\left(e+\left[x_{k+1-i},...,x_{k}\right]\right)^{-1}\right)\\&=&
\pi^{*}\left(e+\left[x_{k+1-i},...,x_{k}\right]\right)
\pi\left(\pi^{-1}(x_{k-i})(e)+\left(e+\left[x_{k+1-i},...,x_{k}\right]\right)^{-1}\right)\\&=&
\pi\left(e+\pi^{*}\left(e+\left[x_{k+1-i},...,x_{k}\right]\right)\pi^{-1}(x_{k-i})(e)\right),
\end{eqnarray*}
where in the last equality, we have used the fact that
$\pi^{*}(x)(x^{-1})=e$. \\
Concerning the term in (\ref{702}), we have
\begin{eqnarray*}
\pi^{-1}(x_{k})\pi\left(e+\left[x_{k}\right]\right)&=&\pi\left(\pi^{-1}(x_{k})\left(e+\left[x_{k}\right]\right)\right)\\
&=&\pi\left(e+\pi^{-1}(x_{k})\left(e\right)\right).
\end{eqnarray*}
Finally, for (\ref{679}), according to the definition of $u_{k}$ see
(\ref{459}), and to Lemma \ref{613}, we have that
$u_{k+1}(e,x_{1},...,x_{k+1})$ can be written in the form
$$u_{k+1}(e,x_{1},...,x_{k+1})=e+\pi^{*}(x_{k+1})(e+v),$$ where $v$
is an element of $\Omega$. Thus
\begin{eqnarray*}
P\left(u_{k+1}(e,x_{1},...,x_{k+1})\right)\pi^{-1}(x_{k+1})&=&P\left(e+\pi^{*}(x_{k+1})(e+v)\right)\pi^{-1}(x_{k+1})\\
&=&\pi\left(e+\pi^{*}(x_{k+1})(e+v)\right)\\&&\pi^{*}\left(e+\pi^{*}(x_{k+1})(e+v)\right)\pi^{-1}(x_{k+1})\\
&=&\pi\left(e+\pi^{*}(x_{k+1})(e+v)\right)\\&&\pi^{*}\left(\pi^{*-1}(x_{k+1})\left(e+\pi^{*}(x_{k+1})(e+v)\right)\right)\\
&=&\pi\left(e+\pi^{*}(x_{k+1})(e+v)\right)\\&&\pi^{*}\left(e+\pi^{*-1}(x_{k+1})(e)+v\right).
\end{eqnarray*}
Having done this, we see that to get $Q_{k}^{*}(y)$, we first apply
$\pi^{-1}(x_{1})$ to $y$, then we apply automorphisms which are of
the form $\pi(e+z)$ or $\pi^{*}(e+z)$ or $P(e+z)$, with
$z\in\Omega$. Using the fact that for all $z$ and $y$ in $\Omega$,
we have that
$$\pi(e+z)(y)-y\in\Omega, \ \pi^{*}(e+z)(y)-y\in\Omega \,  \  \textrm{and} \ \ P(e+z)(y)-y\in\Omega,$$
in other word,
$$\pi(e+z)(y)> y, \  \ \pi^{*}(e+z)(y)> y \, \ \textrm{and} \ \ P(e+z)(y)> y,$$
we conclude that, given $y$ in $\Omega$, we have that for all
$k\geq2$,
$$Q_{k}^{*}(y)>\pi^{-1}(x_{1})(y).$$
\end{proof}
\begin{corollary}For all  $y$ in $\Omega$, there exists a positive
random variable $C$ such that for all $k$,
$$\|Q_{k}^{*}(y)\|>C$$
\end{corollary}
\begin{proof}
As $Q_{k}^{*}(y)>\pi^{-1}(x_{1})(y)$, then
$(Q_{k}^{*}(y))^{2}>(\pi^{-1}(x_{1})(y))^{2}$. In fact
$$(Q_{k}^{*}(y))^{2}-(\pi^{-1}(x_{1})(y))^{2}=(Q_{k}^{*}(y)-\pi^{-1}(x_{1})(y))
(Q_{k}^{*}(y)+\pi^{-1}(x_{1})(y))\in\Omega.$$ Therefore
$\textrm{tr}((Q_{k}^{*}(y))^{2})>\textrm{tr}((\pi^{-1}(x_{1})(y))^{2})$,
which means that $\|Q_{k}^{*}(y)\|>\|\pi^{-1}(x_{1})(y)\|$. Thus it
suffices to take $C=\|\pi^{-1}(x_{1})(y)\|$.
\end{proof}\\
Besides the general facts that we have established above, for the
proof of the main result of the paper concerning the convergence of
a random non ordinary continued fraction, we need the following
probability result which also appears in $[1]$. We give it with a
slightly different proof.
\begin{proposition}\label{768} Let $(\mathfrak{A}_{k})_{k\geq1}$
be an increasing sequence of $\sigma-$fields, and let $A_{k}$ be in
$\mathfrak{A}_{k}$, for $k\geq1$. Suppose that there exit $K>1$,
$0<k_{0}<K$, and a random variable $\alpha$ which is
$\mathfrak{A}_{k_{0}}-$measurable and valued in $(0,1)$ such that
for all $k\geq K $,
$\mathbb{E}\left(1_{A_{k+1}}|\mathfrak{A}_{k}\right)\leq \alpha$. Then\\
$$P\left(\underset{k\geq K }{\bigcap}A_{k}\right)=0$$
\end{proposition}
\begin{proof}Let us show by induction on $p\geq1$ that
$$\mathbb{E}\left(1_{A_{K}}...1_{A_{K+p}}|\mathfrak{A}_{k_{0}}\right)\leq \alpha^{p }\ \ a.s.$$
For $p=1$,
\begin{eqnarray*}
\mathbb{E}\left(1_{A_{K}}1_{A_{K+1}}|\mathfrak{A}_{k_{0}}\right)&=&
\mathbb{E}\left(\left(1_{A_{K}}1_{A_{K+1}}|\mathfrak{A}_{K}\right)|\mathfrak{A}_{k_{0}}\right)\\
&=&\mathbb{E}\left(1_{A_{K}}\left(1_{A_{K+1}}|\mathfrak{A}_{K}\right)|\mathfrak{A}_{k_{0}}\right)
\\
&\leq& \alpha\mathbb{E}\left(1_{A_{K}}|\mathfrak{A}_{k_{0}}\right)\
\ a.s.\\
&\leq& \alpha \ \ a.s..
\end{eqnarray*}
Now suppose that
$\mathbb{E}\left(1_{A_{K}}...1_{A_{K+p}}|\mathfrak{A}_{k_{0}}\right)\leq
\alpha^{p }\ \ a.s.$. Then
\begin{eqnarray*}
\mathbb{E}\left(1_{A_{K}}...1_{A_{K+p+1}}|\mathfrak{A}_{k_{0}}\right)&=&
\mathbb{E}\left(\left(1_{A_{K}}...1_{A_{K+p+1}}|\mathfrak{A}_{K+p}\right)|\mathfrak{A}_{k_{0}}\right)\\
&=&\mathbb{E}\left(1_{A_{K}...1_{A_{K+p}}}\left(1_{A_{K+p+1}}|\mathfrak{A}_{K+p}\right)|\mathfrak{A}_{k_{0}}\right)
\\
&\leq&
\alpha\mathbb{E}\left(1_{A_{K}}...1_{A_{K+p}}|\mathfrak{A}_{k_{0}}\right)\
\ a.s.\\
&\leq& \alpha \alpha^{p }=\alpha^{p+1 }\ \ a.s..
\end{eqnarray*}
Having shown that
$$\mathbb{E}\left(1_{A_{K}}...1_{A_{K+p}}|\mathfrak{A}_{k_{0}}\right)\leq \alpha^{p }\ \ a.s.,$$
we just need to take the expectation and let $p\rightarrow\infty$ to
get the result.
\end{proof}\\
We are now in position to state and prove the main result of the
paper.
\begin{theorem} \label{830}Let $(x_{k})_{k\geq1}$ be a sequence of independent random
variables in the cone $\Omega$ with distributions absolutely
continuous with respect to the Lebesgue measure and such that for
all $k$, the complementary in $\Omega \ $ of the support of
$x_{k}^{-1}$ is bounded. Suppose that there exists $p\geq1$, such
that for all $i\geq1$,
$$ \mathfrak{L}(x_{i})=\mathfrak{L}(x_{i+p}).$$
Then the continued fraction $\left[x_{1},...,x_{k}\right]$ is almost
surely convergent.
\end{theorem}
\begin{proof}With the results what we have established above, we
need only to show that under these assumptions, the sequence
$(w_{k})$ defined in (\ref{461}) converges almost surely to $0$, or
equivalently that for all $y$ in $\Omega$, the sequence $(\langle
w_{k}^{-1}, y \rangle)_{k\geq1}$ diverges almost surely. in order to
do so, we will show that
$$P\left(\left\{\langle w_{k+1}^{-1}-w_{k}^{-1},y\rangle \underset{k\rightarrow \infty }{\rightarrow }0\right\}\right)=0$$
Let  $\mathfrak{A}_{k}$ be the $\sigma-$field generated by the
random variables $x_{1},...,x_{k+1}$, and define for
$\varepsilon>0$,
$$A_{k}=\left\{\langle w_{k}^{-1}-w_{k-1}^{-1},y\rangle \leq\varepsilon\right\}.$$
Then $A_{k}\in \mathfrak{A}_{k}$, and we have
$$\left\{\langle w_{k+1}^{-1}-w_{k}^{-1},y\rangle \underset{k\rightarrow \infty }{\rightarrow }0\right\}\subset
\underset{n>0 }{\bigcup } \ \underset{k\geq n }{\bigcap }A_{k}.$$
Given that $w_{k+1}^{-1}-w_{k}^{-1}= Q_{k}(x_{k+2}^{-1})$, with
$Q_{k}$ defined in (\ref{624}) we have
\begin{eqnarray*}
A_{k+1}&=&\left\{\langle Q_{k}(x_{k+2}^{-1}),y\rangle \leq\varepsilon\right\}\\
&=&\left\{\langle
x_{k+2}^{-1},Q_{k}^{*}(y)\rangle\leq\varepsilon\right\}\\
&=&\left\{\langle
x_{k+2}^{-1},\frac{Q_{k}^{*}(y)}{\|Q_{k}^{*}(y)\|}\rangle\leq\frac{\varepsilon}{\|Q_{k}^{*}(y)\|}\right\}.
\end{eqnarray*}
Let $S$ be the unit sphere in $V$, and define for a random variable
$z$,
$$h_{z}(a)= \underset{\theta\in S}{\sup }P\left(|\langle
z,\theta\rangle|\leq a\right),$$ then we have that
\begin{eqnarray*}
\mathbb{E}\left(1_{A_{k+1}}|\mathfrak{A}_{k}\right)&=&P\left(\left\{\langle
x_{k+2}^{-1},\frac{Q_{k}^{*}(y)}{\|Q_{k}^{*}(y)\|}
\rangle\leq\frac{\varepsilon}{\|Q_{k}^{*}(y)\|}\right\}|\mathfrak{A}_{k}\right)\\&\leq&
h_{x_{k+2}^{-1}}\left(\frac{\varepsilon}{\|Q_{k}^{*}(y)\|}\right), \
\textrm{because} \ Q_{k}^{*}(y) \ is \
\mathfrak{A}_{2}-\textrm{measurable}\\&\leq&
h_{x_{k+2}^{-1}}\left(\frac{\varepsilon}{C}\right).
\end{eqnarray*}
 Denoting
$\alpha=\max
\left(h_{x_{1}^{-1}}\left(\frac{\varepsilon}{C}\right),...,h_{x_{p}^{-1}}
\left(\frac{\varepsilon}{C}\right)\right),$ then according to the
assumptions on the distributions of the $x_{i}$ and on the support
of $x_{i}^{-1}$, we have that $\alpha<1$, and that
$$\mathbb{E}\left(1_{A_{k+1}}|\mathfrak{A}_{k}\right)\leq \alpha.$$
Therefore using Proposition \ref{768} with $k_{0}=2$, we deduce that
$$P\left(\underset{k\geq n }{\bigcap }A_{k}\right),$$ which is the
desired result.
\end{proof}\\
\textbf{Application} Consider the beta distribution of the second
kind $\beta^{(2)}_{p, q} $ on $ \Omega$ given by
$$\beta^{(2)}_{p, q}(dx)=(B_{\Omega}(p, q))^{-1}
(\Delta(x))^{p-\frac{r+1}{2}} (\Delta(e+x))^{-(p+q)}
{\mathbf{1}}_{\Omega }(x) dx,
$$
where  $ B_{\Omega}(p, q)$ is the multivariate beta function defined
by $$ \ B_{\Omega}(p, q)=\frac{\Gamma_{\Omega}(p) \Gamma_{\Omega}(q)
}{\Gamma_{\Omega}(p+q)}, \ \ \textrm{with} \ \
\Gamma_{\Omega}(p)=(2\pi)^{\frac{r(r-1)}{4}}\prod_{k=1}^{r}\Gamma(p-(k-1)/2).$$
Let $(W_n)_{n\geq 1}$ and $(W'_n)_{n\geq 1}$ be two independent
sequences of independent random variables with respective
distributions $\beta_{b,a}^{(2)}$ and $\beta_{b,a'}^{(2)}$, and
consider the random continued fraction
$\left[x_{1},...,x_{k}\right]$ as defined in (\ref{384}) such that
$$x_{2n+1}=W_{n+1}\ \ \textrm{for}\ \ n\geq0 \ \ \textrm{and} \ \ x_{2n}=W'_n  \ \ \textrm{for}\ \ n\geq1.$$
Then we have that \\
$\bullet$  $(x_{n})_{n\geq1}$ is a sequence of independent random
variables in the cone $\Omega$ with distributions absolutely
continuous with respect to the Lebesgue measure. \\
$\bullet$ For all $k$, the support of $x_{k}^{-1}$ is $\Omega $,
its complementary in  $\Omega$ is bounded. \\
$\bullet$ For all $i\geq1$, $
\mathfrak{L}(x_{i})=\mathfrak{L}(x_{i+2}).$ \\According to Theorem
\ref{830}, the continued fraction $\left[x_{1},...,x_{k}\right]$ is
almost surely convergent. As the convergence of a continued fraction
doesn't depend on the initial vector, we deduce that the continued
fraction $\left[e,x_{1},...,x_{k}\right]$ is also almost surely
convergent. This result is used in a forthcoming paper concerning
the matrix beta-hypergeometric distribution.
\noindent \textbf{References}\\

$[1]$ E. Bernadac, Random continued fractions and inverse Gaussian
distribution

on a symmetric cone, J. Theoret. Probab. 8 (1995) 221-256.

$[2]$ J. Faraut, A. Kor\'{a}nyi, Analysis on symmetric cones,

Oxford Univ, Press, (1994).

$[3]$ A. Hassairi, S. Lajmi, Riesz exponential families on symmetric
cones,

J. Theoret. Probab. 14 (2001) 927-948.

$[4]$ A. Kacha, M. Raissouli, Convergence of matrix continued
fractions,

Linear Algebra Appl. 320 (2000) 115-129.

$[5]$ H.-xi Zhaoa, Matrix-valued continued fractions,

J. Approx. Theory 120 (2003) 136-152.
\end{document}